\documentclass[reqno,11pt]{amsart}

\usepackage{amsmath,amsfonts,amssymb,amsthm,epsfig,amscd,}
\usepackage[]{fontenc}
\usepackage[latin1]{inputenc}
\usepackage{enumerate}
\usepackage[cmtip,all]{xy}
\usepackage{tabularx,multirow}
\usepackage{relsize}
\usepackage{layout}
\usepackage{fullpage}
\usepackage{color}
\usepackage{tikz-cd}
 \allowdisplaybreaks \numberwithin{equation}{section}
\newtheorem{theorem}{Theorem}[section]
\newtheorem{proposition}[theorem]{Proposition}

\newtheorem{lemma}[theorem]{Lemma}

\newtheorem{claim}[theorem]{Claim}

\theoremstyle{definition}

\newtheorem{remark}[theorem]{Remark}

\newtheorem{example}[theorem]{Example}

\DeclareMathOperator{\irr}{irr}

\DeclareMathOperator{\conngon}{conn.gon}
\DeclareMathOperator{\covgon}{cov.gon}
\DeclareMathOperator{\gon}{gon}

\DeclareMathOperator{\codim}{codim}

\DeclareMathOperator{\Pol}{Pol}

\DeclareMathOperator{\mult}{mult}

\begin{document}

\title[Cones of lines having high contact with general hypersurfaces]{Cones of lines having high contact with general hypersurfaces and applications}

\author{Francesco Bastianelli}
\address{Francesco Bastianelli, Dipartimento di Matematica, Universit\`{a} degli Studi di Bari Aldo Moro, Via Edoardo Orabona 4, 70125 Bari -- Italy}
\email{francesco.bastianelli@uniba.it}

\author{Ciro Ciliberto}
\address{Ciro Ciliberto, Dipartimento di Matematica, Universit\`{a} degli Studi di Roma ``Tor Vergata", Viale della Ricerca Scientifica 1, 00133 Roma -- Italy}
\email{cilibert@mat.uniroma2.it}

\author{Flaminio Flamini}
\address{Flaminio Flamini, Dipartimento di Matematica, Universit\`{a} degli Studi di Roma ``Tor Vergata", Viale della Ricerca Scientifica 1, 00133 Roma -- Italy}
\email{flamini@mat.uniroma2.it}

\author{Paola Supino}
\address{Paola Supino, Dipartimento di Matematica e Fisica, Universit\`{a} degli Studi ``Roma Tre", Largo S. L. Murialdo 1, 00146 Roma -- Italy}
\email{supino@mat.uniroma3.it}

\begin{abstract}
Given a smooth hypersurface $X\subset \mathbb{P}^{n+1}$ of degree $d\geqslant 2$, we study the cones $V^h_p\subset \mathbb{P}^{n+1}$ swept out by lines having contact order $h\geqslant 2$ at a point $p\in X$.
In particular, we prove that if $X$ is general, then for any $p\in X$ and $2 \leqslant h\leqslant \min\{ n+1,d\}$, the cone $V^h_p$ has dimension exactly $n+2-h$.
Moreover, when $X$ is a very general hypersurface of degree $d\geqslant 2n+2$, we describe the relation between the cones $V^h_p$ and the degree of irrationality of $k$--dimensional subvarieties of $X$ passing through a general point of $X$. 
As an application, we give some bounds on the least degree of irrationality of $k$--dimensional subvarieties of $X$ passing through a general point of $X$, and we prove that the connecting gonality of $X$ satisfies $d-\left\lfloor\frac{\sqrt{16n+25}-3}{2}\right\rfloor\leqslant\conngon(X)\leqslant d-\left\lfloor\frac{\sqrt{8n+1}+1}{2}\right\rfloor$.
\end{abstract}

\thanks{This collaboration has benefitted of funding from the MIUR Excellence Department Project awarded to the Department of Mathematics,
University of Rome Tor Vergata (CUP: E83-C18000100006).\\
\indent The authors are members of GNSAGA of INdAM}

\maketitle

\section{Introduction}\label{sec1}

Let $X \subset \mathbb{P}^{n+1}$ be a smooth complex hypersurface of degree $d\geqslant 2$. 
Given a point $p\in X$ and an integer $h\geqslant 2$, we consider the cone $V^h_p\subset \mathbb{P}^{n+1}$ swept out by lines having intersection multiplicity at least $h$ with $X$ at $p$.
These cones reflect the geometry of hypersurfaces and occur both in the local geometry of hypersurfaces (see e.g. \cite{GH, J1, J2}) and in the study of their global properties, such as unirationality (\cite{Ci}) and their covering gonality, i.e. the least gonality of curves passing through a general point of $X$ (\cite{BCFS}).

In this paper, we study the cones $V^h_p$ of general hypersurfaces $X \subset \mathbb{P}^{n+1}$, and we apply our results to achieve some bounds concerning the degree of irrationality of $k$--dimensional subvarieties of $X$ passing through general points of $X$, where we recall that the degree of irrationality of an irreducible variety $Y$ of dimension $k$ is the least degree of a dominant rational map $Y\dashrightarrow \mathbb{P}^k$.  

\smallskip
Given a smooth hypersurface $X \subset \mathbb{P}^{n+1}$ of degree $d\geqslant 2$, a point $p\in X$, and an integer $2\leqslant h\leqslant d$, the cone $V^h_p\subset \mathbb{P}^{n+1}$ is defined by the vanishing of $h-1$ polynomials of degree $1,2,\dots,h-1$, respectively, where the linear polynomial defines the tangent hyperplane of $X$ at $p$ (cf. Section \ref{section:cones}).
When $p\in X$ is a general point, then $V^h_p$ is a complete intersection defined by those polynomials, i.e. $\dim V^h_p=n+2-h$ (cf. \cite{J1}).
However, it may happen that for some special point of $X$, the cone $V^h_p$ fails to be a complete intersection of multi--degree $(1,2,\dots,h-1)$ and its dimension is larger than expected.
We prove that when $X \subset \mathbb{P}^{n+1}$ is a general hypersurface of degree $d\geqslant 2$, this is not the case.
\begin{theorem}\label{theorem:dimV^h_p}
Let $n \geqslant 2$ be an integer and let $X\subset \mathbb{P}^{n+1}$ be a general hypersurface of degree $d\geqslant 2$. 
Then, for any point $p\in X$ and for any integer $2 \leqslant h\leqslant \min\{ n+1,d\}$, the cone $V^h_p$ has dimension 
$$\dim(V^h_p)  = n+2-h.$$
\end{theorem} 

\smallskip
In recent years there has been a great deal of interest concerning measures of irrationality of projective varieties, that is birational invariants which somehow measure the failure of a given variety to be rational (see e.g. \cite{BDELU,Ch,M,CMNP,GK,SU,V}), and several interesting results have been obtained in this direction for very general hypersurfaces of large degree (cf. \cite{BCD, BDELU, BCFS, Y}).

Given an irreducible complex projective variety $X$ of dimension $n$ and an integer $k$ such that $1\leqslant k\leqslant n$, we are interested in the following birational invariants.
According to \cite[Section 5.3]{BCFS}, we define the $k$--\emph{irrationality degree} of $X$ as the integer
\begin{equation*}
\irr_k(X):=\min\left\{c\in \mathbb{N}\left|
\begin{array}{l}
\text{Given a general point} \; p\in X, \, \exists \text{ an irreducible subvariety } Z\subseteq X \\ \text{of dimension } k \text{ such that } p\in Z \text{ and there is a dominant rational} \\ \text{map }  Z\dashrightarrow \mathbb{P}^k \text{ of degree } c
\end{array}\right.\right\}
\end{equation*}and, in line with \cite{BDELU}, we define the \emph{connecting gonality} of $X$ as the integer
\begin{equation*}
\conngon(X):=\min\left\{c\in \mathbb{N}\left|
\begin{array}{l}
\text{Given two general points }q,q'\in X,\,\exists\text{ an irreducible}\\ \text{curve } C\subset X  \text{ such that }q,q'\in C \text{ and }\gon(C)=c
\end{array}\right.\right\}.
\end{equation*}
Therefore, $\conngon(X)$ can be thought as a measure of the failure of $X$ to be rationally connected, whereas $\irr_k(X)$ measures how $X$ is far from being covered by $k$--dimensional rational varieties. 
We note further that $\irr(X):=\irr_n(X)$ is the \emph{degree of irrationality} of $X$ and $\covgon(X):=\irr_1(X)$ is the \emph{covering gonality} of $X$.
Moreover, these invariants satisfy the obvious inequalities
\begin{equation}\label{eq:inequalities}
\covgon(X)\leqslant\conngon(X)\leqslant\irr(X) \quad\text{and}\quad \irr_1(X)\leqslant\irr_2(X)\leqslant\dots\leqslant\irr_n(X).
\end{equation}

In \cite{CMNP}, it has been proved that $\irr_k(A)\geqslant k+\frac{1}{2}\left(\dim A+1\right)$, provided that $A$ is a very general abelian variety of dimension at least 3 and $1\leqslant k\leqslant \dim A$. 
Apart from this result and the cases $k\in \{1,n\}$, very little is known about the $k$--irrationality degrees and the connecting gonality of projective varieties. 
When $X\subset \mathbb{P}^{n+1}$ is a very general hypersurface of degree $d\geqslant 2n+2$, it follows from \cite[Theorem C]{BDELU} and \cite[Theorem 1.1]{BCFS} that
\begin{equation}\label{eq:irr&covgon}
\irr_n(X)=d-1 \quad \text{and} \quad d-\left\lfloor\frac{\sqrt{16n+9}-1}{2}\right\rfloor\leqslant \irr_1(X)\leqslant d-\left\lfloor\frac{\sqrt{16n+1}-1}{2}\right\rfloor,
\end{equation}
where the latter relation is often a chain of equalities.
Under the same assumption on $X\subset \mathbb{P}^{n+1}$, we are concerned with its connecting gonality and the $k$--irrationality degree, with $2\leqslant k\leqslant n-1$.

\smallskip
In this direction, we extend \cite[Proposition 2.12]{BCFS}, which relates curves of low gonality covering $X$ to the cones $V^h_p$ of lines having high contact with $X$.
In particular, we prove that if $Z\subset X$ is a $k$--dimensional subvariety passing through a general point, endowed with a dominant rational map $\varphi\colon Z\dashrightarrow \mathbb{P}^k$ of degree $c\leqslant d-3$, then $Z\subset V^{d-c}_p$ for some $p\in X$, and the map $\varphi$ is the projection from $p$ (cf. Proposition \ref{proposition:Cone}).
Combining Theorem \ref{theorem:dimV^h_p} and the latter result, we achieve bounds on the $k$--irrationality degrees of $X$.
\begin{theorem}\label{theorem:irr_k}
Let $n\geqslant 3$ and let $X\subset \mathbb{P}^{n+1}$ be a very general hypersurface of degree $d\geqslant 2n+2$.
Then
\begin{equation}\label{eq:irr_k}
\irr_k(X)\geqslant d-1-n+k \quad\text{for  }1\leqslant k\leqslant n.
\end{equation}
Moreover, equality holds for $n-2\leqslant k\leqslant n$, that is
\begin{equation*}
\irr_{n-2}(X)= d-3, \quad\irr_{n-1}(X)= d-2 \quad\text{and} \quad\irr_{n}(X)= d-1.
\end{equation*}
\end{theorem}
We point out that the assertion for $k=n$ is given by \cite[Theorem C]{BDELU}.
Furthermore, the larger the value of $k$ is, the more significant Theorem \ref{theorem:irr_k} becomes, since for small values of $k$ the bound \eqref{eq:irr_k} is superseded by \eqref{eq:irr&covgon} and \eqref{eq:inequalities}.

\smallskip
In order to discuss the connecting gonality of $X$, we prove further that for any pair of general points $q,q'\in X$ and for any $2\leqslant h\leqslant \frac{n}{2}+1$, there exists a general point $p\in X$ such that $q,q'\in V^{h}_p$ (cf. Lemma \ref{lemma:general}). 
Then, using the fact that for $h=\left\lfloor\frac{\sqrt{8n+1}+1}{2}\right\rfloor$ the locus $V^{h}_p$ is a cone over a rationally connected variety, we bound from above the connecting gonality of a very general hypersurface $X\subset \mathbb{P}^{n+1}$ of large degree.
\begin{theorem}\label{theorem:conngon}
Let $n\geqslant 4$ and let $X\subset \mathbb{P}^{n+1}$ be a very general hypersurface of degree $d\geqslant 2n+2$.
Then
\begin{equation}\label{eq:conngon}
\conngon(X)\leqslant d-\left\lfloor\frac{\sqrt{8n+1}+1}{2}\right\rfloor.
\end{equation}
\end{theorem}

\smallskip
Finally, using our results and the Grassmannian techniques introduced in \cite{RY}, we also obtain a lower bound on the connecting gonality of very general hypersurfaces, which slightly improves the bound descending from \eqref{eq:irr&covgon} and \eqref{eq:inequalities}.
\begin{theorem}\label{theorem:conngonRY}
Let $n\geqslant 4$ and let $X\subset \mathbb{P}^{n+1}$ be a very general hypersurface of degree $d\geqslant 2n+2$.
Then
\begin{equation}\label{eq:conngon2}
\conngon(X)\geqslant d-\left\lfloor\frac{\sqrt{16n+25}-3}{2}\right\rfloor.
\end{equation}
In particular,
\begin{equation*}
\conngon(X)>\covgon(X)
\end{equation*}
$\forall n\in\mathbb{Z}_{\geqslant 4}\smallsetminus\left\{\left.4a^2+3a,4a^2+5a,4a^2+5a+1,4a^2+7a+2,4a^2+9a+4,4a^2+11a+6\right|a\in\mathbb{N}\right\}$.
\end{theorem}
In Example \ref{example:small n} we also discuss the cases $1\leqslant n \leqslant 3$, which turn out to satisfy equality in \eqref{eq:conngon2}. 
We note that the second part of the statement of Theorem \ref{theorem:conngonRY} is obtained from \eqref{eq:irr&covgon} by determining the values of $n$ such that $\left\lfloor\frac{\sqrt{16n+1}-1}{2}\right\rfloor\neq \left\lfloor\frac{\sqrt{16n+25}-3}{2}\right\rfloor$.
We believe that the bound \eqref{eq:conngon2} is far from being sharp.
However, for any $4\leqslant n\leqslant 16$ with $n\neq 9,13,14$, the right--hand sides of \eqref{eq:conngon} and \eqref{eq:conngon2} do coincide, hence Theorems \ref{theorem:conngon} and \ref{theorem:conngonRY} compute the connecting gonality of $X$ in these cases (cf. Example \ref{example:small n}).

\smallskip
The paper is organized as follows.
In Section \ref{section:cones}, we recall some basic facts on the cones of lines of high contact with a general hypersurface $X\subset \mathbb{P}^{n+1}$ and we prove Theorem \ref{theorem:dimV^h_p}.

In Section \ref{section:polar}, we consider polar hypersurfaces of a smooth hypersurface $X\subset \mathbb{P}^{n+1}$, in order to discuss when any pair $q,q'\in X$ lies on a cone $V^h_p$ for some $p\in X$.

Finally, Section \ref{section:bounds} is concerned with the applications to measures of irrationality. 
In particular, after describing the relation between the cones $V^h_p$ and the degree of irrationality of $k$--dimensional subvarieties of $X$ passing through a general point, we prove Theorems \ref{theorem:irr_k}, \ref{theorem:conngon} and \ref{theorem:conngonRY}, and we discuss the behavior of the connecting gonality for small values of $n=\dim X$.

\subsection*{Notation}

We work throughout over the field $\mathbb{C}$ of complex numbers.
By \emph{variety} we mean a complete reduced algebraic variety $X$, and by \emph{curve} we mean a variety of dimension 1.
We say that a property holds for a \emph{general} (resp. \emph{very general}) point ${x\in X}$ if it holds on a Zariski open nonempty subset of $X$ (resp. on the complement of the countable union of proper subvarieties of $X$).

\section{Dimension of cones of lines having high contact}\label{section:cones}

Let $n \geqslant 2$ be an integer, and let $X=V(F)\subset \mathbb{P}^{n+1}$ be a smooth hypersurface defined by the vanishing of a homogeneous polynomial $F\in \mathbb{C}[x_0,\dots,x_{n+1}]$ of degree $d\geqslant 2$.
Given a point $p\in X$ and an integer $2\leqslant h\leqslant d$, we define the \emph{cone} $V^h_p=V^h_{p,X}\subset  \mathbb{P}^{n+1}$ \emph{of tangent lines of order} $h$ \emph{at} $p$ as the Zariski closure of the locus swept out by lines $\ell\subset \mathbb{P}^{n+1}$ such that either $\ell\subset X$ or $\ell\cdot X\geqslant hp$. 
Therefore $V^h_p$ is a cone with vertex at $p$, defined by the vanishing of the following $h-1$ polynomials occurring in the Taylor expansion of $F$ at $p$
\begin{equation}\label{eq:Gk}
\begin{aligned}
G_k(x_0,\ldots,x_{n+1}) & :=\left(x_0\frac{\partial}{\partial x_0}+\cdots+x_{n+1}\frac{\partial}{\partial x_{n+1}}\right)^{(k)}F(p)\\
& =\sum_{i_0+\cdots+i_{n+1}=k} \frac {k!}{i_0!\cdots i_{n+1}!} x_0^{i_0}\cdots 
x_{n+1}^{i_{n+1}}\frac {\partial ^kF}{\partial x_0^{i_0}\cdots \partial 
x_{n+1}^{i_{n+1}}}(p)
\end{aligned}
\end{equation}
where $(-)^{(k)}$ denotes the usual symbolic power, $1\leqslant k\leqslant h-1$, and $\deg G_k=k$  (cf. \cite[p.\,186]{Ci}).
In particular, the cone $V^2_p$ coincides with the (projective) tangent hyperplane $T_pX\subset \mathbb{P}^{n+1}$ to $X$ at $p$. 
When instead $h\geqslant 3$, we denote by $\Lambda_p^{h}=\Lambda^{h}_{p,X}$ the intersection of $V_p^{h}$ with a general hyperplane $H\subset \mathbb{P}^{n+1}$ not containing $p$, so that $\Lambda_p^{h}$ is defined in $H\cap T_pX\cong \mathbb{P}^{n-1}$ by $h-2$ polynomial equations of degree $2,3,\dots,h-1$, respectively, and $V_p^{h}$ is the cone over $\Lambda_p^{h}$ with vertex at $p$.

For any $2\leqslant h\leqslant \min\{n+1,d\}$, it follows from this description of $V_p^{h}$ that
\begin{equation}\label{eq:expdim}
\dim V^h_p\geqslant  n+2-h.
\end{equation}
When $X$ is a general hypersurface of degree $d\geqslant 2$ and $p\in X$ is a general point, \cite[Lemma 2.2]{BCFS} guarantees that $V^h_p\subset \mathbb{P}^{n+1}$ is a general complete intersection of multi--degree $(1,2,\dots,h-1)$ and, in particular, \eqref{eq:expdim} is an equality.

According to the assertion of Theorem \ref{theorem:dimV^h_p}, we want to prove that equality in \eqref{eq:expdim} holds for \emph{any} point $p\in X$. 
The case $h=2$ is trivial since $X$ is smooth and hence $V^2_p=T_pX\cong \mathbb{P}^n$.
The assertion for $h=3$ is implied by the following result.

\begin{lemma}\label{lemma:3ple} Let $n \geqslant 2$ be an integer and let $X\subset \mathbb{P}^{n+1}$ be a general hypersurface of degree $d\geqslant 2$. 
Then, for any point $p\in X$, the intersection of $X$ with the tangent hyperplane $T_{p}X$ has multiplicity 2 at $p$.
\end{lemma}

\begin{proof} Let $\mathcal L$ be the linear system of all hypersurfaces of degree $d$ in $\mathbb P^{n+1}$, which has dimension
\[
\dim\mathcal L={{d+n+1}\choose {n+1}}-1.
\] 
Consider the variety $\mathcal{V}$ consisting of all triples $(p,\Pi,Y)$ where 
$\Pi\subset \mathbb{P}^{n+1}$ is a hyperplane, $p\in \Pi$ and $Y\subset \Pi$ is a hypersurface of degree $d$ with a point of multiplicity at least 3 at $p$. 
Then we define the variety
\[
\mathcal{Z}:=\left\{ \left.(p,\Pi,Y,X) \in \mathcal{V}\times \mathcal L\right| Y\subset X \right\}.
\]
endowed with the projection $\pi_1\colon \mathcal{Z}\longrightarrow \mathcal{V}$, whose fibers are all isomorphic to linear systems of hypersurfaces of degree $d$ of the same dimension 
$
{{d+n}\choose {n+1}}.
$
Looking at the map $\mathcal{V}\longrightarrow \mathbb P^{n+1}\times \left(\mathbb{P}^{n+1}\right)^*$ given by $(p,\Pi,Y)\longmapsto (p,\Pi)$, it is easy to see that $\mathcal{V}$ is irreducible of dimension
\[
\dim\mathcal{V}=2n+ {{d+n}\choose {n}}-{{n+2}\choose {2}}.
\]
Hence $\mathcal{Z}$ is also irreducible, having dimension
\[
\dim\mathcal{Z}=2n+ {{d+n}\choose {n}}-{{n+2}\choose {2}}+{{d+n}\choose {n+1}}=2n+ {{d+n+1}\choose {n+1}}-{{n+2}\choose {2}}.
\]
Consider now the projection $\pi_2\colon \mathcal{Z}\longrightarrow \mathcal L$, whose image is the locus $\mathcal{T}$ of all hypersurfaces $X\subset \mathbb{P}^{n+1}$ of degree $d$ having a point $p\in X$ and a hyperplane $\Pi\subset \mathbb{P}^{n+1}$ such that the intersection scheme $X\cap \Pi$ has a point of multiplicity at least 3 at $p$. 
Hence
\[
\dim \mathcal{T}\leqslant \dim \mathcal{Z} = %\left ({{d+n+1}\choose {n+1}}-1\right) +\left(2n-{{n+2}\choose {2}}+2\right)=
\dim \mathcal L+2n+1-{{n+2}\choose {2}}
\]
and, since ${{n+2}\choose {2}}>2n+1$ as soon as $n\geqslant 2$, we conclude that $\mathcal{T}$ is a proper closed subset of $\mathcal L$, as wanted. 
\end{proof}

We notice that the double point at $p\in X$ of the intersection of $X$ with $T_pX$ as in Lemma \ref{lemma:3ple} does not need to be an ordinary double point, and the locus where the singularity is worse than an ordinary double point is the intersection of $X$ with its Hessian hypersurface.

\smallskip
Now we prove Theorem \ref{theorem:dimV^h_p}.

\begin{proof}[Proof of Theorem \ref{theorem:dimV^h_p}]

We already discussed the trivial case $h=2$. 
Moreover, Lemma \ref{lemma:3ple} ensures that for any $p\in X$, a general line $\ell\subset T_pX$ tangent to $X$ at $p$ intersects $X$ at $p$ with multiplicity exactly 2.
Hence $V^3_p$ is a proper subvariety of $T_pX$, and \eqref{eq:expdim} is an equality.
Thus we assume hereafter $h\geqslant 4$.

Let $[x_0:\ldots: x_{n+1}]$ be homogeneous coordinates in ${\mathbb P}^{n+1}$ and, for any positive integer $k$, we set
\begin{equation}
\label{eq:S_k}
S_k := \mathbb{C} [x_0, \ldots, x_{n+1}]_k\quad \text{and} \quad S_k^* := \mathbb{C} [x_0, \ldots, x_{n+1}]_k \setminus \{0\}.
\end{equation}
For any $F \in S_d^*$, we denote by $V(F) \subset \mathbb{P}^{n+1}$ the hypersurface defined by the vanishing of $F$, and for any $\mathbf{G}:= \left(G_1, \ldots, G_{h-1}\right) \in \prod_{k=1}^{h-1} S_k$, we denote by $V(\mathbf{G})$ the intersection scheme of the hypersurfaces $V(G_k)$ for  $1\leqslant k\leqslant h-1$.

For $F \in S_d^*$ and $p \in V(F)\subset \mathbb{P}^{n+1}$, let $G_k=G_{p,k}(F)(x_0,\ldots,x_{n+1})$ be the homogeneous polynomial of degree $k$ defined in \eqref{eq:Gk}, where $1\leqslant k\leqslant h-1$, and let 
$$
\mathbf{G}=\mathbf{G}_p(F) := \left(G_{p,1}(F),\dots,G_{p,h-1}(F)\right)\in \prod_{k=1}^{h-1} S_k.
$$
Therefore $V_p^h$ is the cone $V(\mathbf{G}_p(F))$ with vertex at $p$, and \eqref{eq:expdim} fails to be an equality if and only if $\mathbf{G}_p(F)$ is not a \emph{regular sequence}. 
In order to prove the assertion, we show that if $F\in S_d^*$ is general, then the sequence $\mathbf{G}_p(F)$ is regular for all points $p\in V(F)$. 

To this aim, let $U_d \subset S_d^*$ be the open dense subset parametrizing those $F \in S_d^*$ such that $V(F)$ is smooth and let
$$
\mathcal{J}:=\left\{ \left.\left(p , F, {\mathbf{G}}\right) \in \mathbb{P}^{n+1} \times U_d \times \prod_{k=1}^{h-1} S_k \right| p\in V(F)\text{ and }\mathbf{G} =\mathbf{G}_p(F) \right\},
$$
which is endowed with the two projections $\pi_1\colon \mathcal{J}\longrightarrow U_d$ and $\pi_2\colon \mathcal{J}\longrightarrow \mathbb{P}^{n+1} \times \prod_{k=1}^{h-1} S_k$.
The map $\pi_1$ is surjective, and for any $F\in U_d$, the fiber $\pi_1^{-1}(F)$ is isomorphic to $V(F)$, which is irreducible of dimension $n$. 
Thus $\mathcal{J}$ is irreducible of dimension
\begin{equation}\label{eq:dimJ}
\dim  \mathcal{J} =  \dim (U_d) + n = {n+1+d\choose d} +n.
\end{equation}
Let us define $\mathcal{W} := \pi_2(\mathcal{J}) \subset \mathbb{P}^{n+1} \times \prod_{k=1}^{h-1} S_k$, which is irreducible too. 
\begin{claim}\label{claim:codimW1} All fibers of $\pi_2\colon \mathcal{J}\longrightarrow \mathcal{W}$ have dimension 
\[
f={{n+h}\choose h}+{{n+h+1}\choose {h+1}}+\cdots+{{n+d}\choose d}={{n+d+1}\choose {d}}-{{n+h}\choose {h-1}}.
\]
\end{claim}

\begin{proof} [Proof of Claim \ref {claim:codimW1}] Let $(p,\mathbf{G})=\pi_2(p,F, \mathbf{G})\in \mathcal{W}$, with $(p,F, \mathbf{G})\in \mathcal{J}$, $\mathbf{G}=(G_1,\dots, G_{h-1})$ and $G_k=G_{p,k}(F)$ for $1\leqslant k\leqslant h-1$. 
Up to projective transformations, we may assume $p=[1:0: \ldots:0]$ and $T_pV(F) = V(x_{n+1})$. 
Then $F$ is of the form
\begin{equation}\label{eq:F}
F(x_0,\ldots, x_{n+1})=cx_{n+1}x_0^{d-1}+F_2(x_1,\ldots,x_{n+1})x_0^{d-2}+\dots +F_d(x_1,\ldots, x_{n+1}),
\end{equation}
where $c$ is a non--zero constant and each $F_i\in \mathbb{C}[x_1,\dots,x_{n+1}]$ is homogeneous of degree $i$. 
Easy computations show that
\begin{equation}\label{eq:form}
\begin{aligned}
G_{1}=& cx_{n+1}\\
G_{2}=& 2(d-1)cx_{n+1}x_0+ 2F_2 \\
G_{3}=& 3(d-1)(d-2)cx_{n+1}x^2_0+ 6(d-2)x_0F_2+6F_3 \\
\dots\\
G_{h-1}=& (h-1)\frac{(d-1)!}{(d-h+1)!}cx_{n+1}x_0^{h-2}+\sum_{i=2}^{h-1}\frac{(h-1)!}{(h-i-1)!}\frac{(d-i)!}{(d-h+1)!}x_0^{h-i-1}F_i.
\end{aligned}
\end{equation}
To determine the fiber of $\pi_2$ over $(p, \mathbf{G})$, we have to find all forms $F'\in U_d$ such that $(p,F', \mathbf{G})\in \mathcal{J}$. 
As in \eqref {eq:F}, we have $F'=c'x_{n+1}x_0^{d-1}+F'_2x_0^{d-2}+\dots +F'_d$, which satisfies the corresponding equations in \eqref{eq:form}.
Therefore, Equations \eqref{eq:form} imply $c=c'$ and $F_k'=F_k$ for any $2\leqslant k\leqslant h-1$.
Thus $F'$ may differ from $F$ by the terms $F'_i$ for $h\leqslant i\leqslant d$, which can be chosen arbitrarily in $\mathbb{C}[x_1,\dots,x_{n+1}]_i$.
\end{proof}

We deduce from the claim and \eqref{eq:dimJ} that
$$
\dim \mathcal{W}=\dim \mathcal{J} - f={{n+h}\choose {h-1}}+n.
$$
Moreover, according to the description of \eqref{eq:form}, this equality implies that $\mathcal W$ coincides with the set of all $h$--tuples $(p, G_1,...,G_{h-1})$ where $p\in \mathbb{P}^{n+1}$ is arbitrary, $G_1$ is an arbitrary homogeneous polynomial of degree 1 vanishing at $p$, and for any $k=2,\dots,h-1$, $G_k$ is a homogeneous polynomial of degree $k$ such that modulo $G_1,\dots,G_{k-1}$ is arbitrary.

Then we set 
$$
\mathcal{W}_0 := \left\{ \left.(p, {\mathbf{G}}) \in \mathcal{W}  \right| \mathbf{G} \text{ is not a regular sequence}\right\} \subset \mathcal{W},
$$
so that proving the assertion is equivalent to showing that $\pi_1\left(\pi_2^{-1} (\mathcal{W}_0)\right)$ is a proper closed subset of $U_d$.
In particular, it suffices to prove that
\begin{equation}\label{eq:codimW0}
\codim_\mathcal{W} \mathcal{W}_0 > n,
\end{equation} 
because in this case Claim \ref{claim:codimW1} and \eqref{eq:dimJ} yield
\begin{align*}
\dim  \pi_1\left(\pi_2^{-1} (\mathcal{W}_0)\right) &\leqslant \dim \pi_2^{-1} (\mathcal{W}_0) = \dim \mathcal{W}_0 + f = \dim\mathcal{W}_0 + \dim \mathcal{J} - \dim \mathcal{W} =\\
&=\dim \mathcal{J} - \codim_\mathcal{W} \mathcal{W}_0 < \dim \mathcal{J} - n = \dim U_d,
\end{align*}
so that the assertion holds. Hence we focus on proving \eqref {eq:codimW0}.

To this aim, let $\mathcal{Z}$ be an irreducible component of $\mathcal{W}_0$ and let $(p, {\mathbf{G}})\in \mathcal{Z}$ be a general point, with $\mathbf{G}=(G_1,\dots, G_{h-1})$.
Then $\mathbf{G}$ is not a regular sequence, and we may define $\alpha=\alpha_\mathcal{Z}$ to be the greatest integer such that $(G_1,\dots, G_{\alpha})$ is a regular sequence. 
Since we already showed that $V^3_p=V(G_1,G_2)$ has dimension $n-1$, we have $2\leqslant \alpha< h-1$.
By maximality of $\alpha$, there exists an irreducible component  $Y$ of  $V(G_1,G_2,\dots,G_{\alpha})$ of dimension $m: = \dim Y = n -\alpha+1$ and contained in $V(G_{\alpha+1})$. Therefore, in view of the above interpretation of $\mathcal{W}$ in terms of $h$--tuples $(p, G_1,...,G_{h-1})$, in order to prove \eqref{eq:codimW0}, it is enough to show that the \emph{Hilbert function} $h_Y\colon \mathbb{N}\longrightarrow \mathbb{N}$ of $Y$ satisfies $h_{Y}(\alpha+1)>n$.

For any $\ell=0,\dots, m$, let $Y_{\ell}$ denote a general linear section of $Y$ of dimension $\ell$. 
In particular, $Y_m=Y$ and $Y_{m-1}$ is a general hyperplane section of $Y$. 
Then, for any positive integer $t$, we have
\begin{equation*}
h_Y(t) - h_Y(t-1) \geqslant h_{Y_{m-1}}(t)
\end{equation*}
(see \cite[Lemma 3.1]{H}). 
Analogously, for any $1\leqslant \ell\leqslant m$ and for any integer $t>0$, we obtain
\begin{equation}\label{eq:Ciro*}
h_{Y_{\ell}} (t) \geqslant h_{Y_{\ell}}(t-1) + h_{Y_{\ell-1}}(t),
\end{equation}
and we deduce by iteration that for any integer $t>0$,
\begin{equation}\label{eq:Ciro**}
h_Y(t) \geqslant \sum_{\ell =1}^m h_{Y_{\ell}} (t-1) + h_{Y_0}(t) \geqslant  \sum_{\ell =1}^m h_{Y_{\ell}} (t-1) + 1.
\end{equation}

\begin{claim}\label{claim:h_Y}
For any $1 \leqslant \ell \leqslant m$ and for any integer $t>0$,
$$
h_{Y_{\ell}} (t) \geqslant {\ell +t \choose t}.
$$
\end{claim}

\begin{proof}[Proof of Claim \ref{claim:h_Y}]
We recall that $Y_{\ell}$ is a general linear section of $Y\subset V(G_1)\cong \mathbb{P}^n$ of dimension $\ell$, so that $Y_{\ell}$ is irreducible and it sits in a projective space $\Lambda_{\ell} \cong {\mathbb P}^{n+\ell-m}$, with $1 \leqslant \ell \leqslant m$. 
If $t=1$, the claim is true as the linear span of $Y_\ell\subset \Lambda_{\ell}$ has dimension at least $\ell=\dim Y_{\ell}$, i.e. $Y_{\ell}$ contains at least $\ell+1$ independent points of $\Lambda_{\ell}$.
Then we argue by induction on $t$, and using inequality \eqref{eq:Ciro**} applied to $Y_{\ell}$, we obtain
$$
h_{Y_{\ell}}(t) \geqslant \sum_{j= 1}^{\ell} h_{Y_{j}} (t-1) + 1 \geqslant \sum_{j= 1}^{\ell} {j+t-1 \choose t-1} + 1 = \sum_{j= 0}^{\ell} {j+t-1 \choose t-1} =
{\ell+t \choose t}
$$
as desired.
\end{proof}

Finally, setting $t = \alpha+1$ and $\ell = m = n-\alpha+1$, Claim \ref{claim:h_Y} ensures that
$$
h_Y(\alpha+1) \geqslant {m + \alpha + 1 \choose \alpha+1} = {n+ 2 \choose \alpha+1}.
$$
By assumption, we have $\alpha < h-1 \leqslant n$, and hence ${n+ 2 \choose \alpha+1} \geqslant n+1$.
Thus $h_Y(\alpha+1)>n$, which concludes the proof of Theorem \ref{theorem:dimV^h_p}.
\end{proof}

\section{Polar hypersurfaces and cones of lines having high contact}\label{section:polar}
 
Let $n \geqslant 2$ be an integer and let $X := V(F)\subset \mathbb{P}^{n+1}$ be a smooth hypersurface of degree $d\geqslant 2$. 
Given a point $q= [q_0:\ldots:q_{n+1}] \in \mathbb P^{n+1}$ and an integer $0\leqslant s\leqslant d$, we introduce the \emph{$s$--th polar hypersurface of $X$  with respect to $q$} as the hypersurface $\Delta^s_q=\Delta^s_q(X)\subset \mathbb P^{n+1}$ defined by the vanishing of the polynomial of degree $d-s$
\begin{equation}
\Pol_{q}^s(F)(x_0,\ldots,x_{n+1}):= \left(q_0\frac{\partial}{\partial x_0}+\cdots+q_{n+1}\frac{\partial}{\partial x_{n+1}}\right)^{(s)}F(x_0,\ldots,x_{n+1}),
\end{equation}
where $(-)^{(s)}$ denotes the usual symbolic power and $\Pol_{q}^0(F)=F$, that is $\Delta^0_q=X$ for any $q\in \mathbb{P}^{n+1}$.
Furthermore, we define the intersection scheme
\begin{equation}\label{eq:Delta}
\Delta_{q,h}(X):= \bigcap_{s=0}^{h-1}\Delta^s_q.
\end{equation}

In this section, we use polar hypersurfaces of $X$ and Theorem \ref{theorem:dimV^h_p} to show that for general $q,q'\in X$ and for any $2\leqslant h\leqslant \frac{n}{2}+1$, there exists a general point $p\in X$ such that $q,q'\in V^{h}_p$.
In particular, this fact is crucial in order to prove Theorem \ref{theorem:conngon}.  
To start, we prove the following.

\begin{lemma}\label{lemma:polari} 
Let $n \geqslant 2$ be an integer and let $ X \subset \mathbb{P}^{n+1}$ be a smooth hypersurface of degree $d \geqslant 2$.
For any integer $2 \leqslant h \leqslant \frac{n}{2} +1$ and for any $q, q'\in X$, there exists a point $p\in X$ such that $q, q'\in V_p^{h}$.
\end{lemma}

\begin{proof}  
We point out that for $q\in \mathbb{P}^{n+1}$, the intersection $X\cap \Delta^1_q$ consists of the points $p \in X$ such that $q\in T_{p}X$, i.e. the line $\langle q,p \rangle$ intersects $X$ with multiplicity at least 2 at $p$, provided that $p\neq q$.
Similarly, given two points $q\in \mathbb{P}^{n+1}$ and $p\in X$ with $p\neq q$, the line $\langle q,p \rangle$ intersects $X$ with multiplicity at least $h$ at $p$---that is $q\in V^h_p$---if and only if $p$ belongs to $\Delta_{q,h}(X)$ defined in \eqref{eq:Delta}.

Therefore, proving the statement is equivalent to showing that for any $q,q'\in X$, the intersection of $\Delta_{q,h}(X)$ and $\Delta_{q',h}(X)$ is not empty.
Since both $\Delta_{q,h}(X)$ and $\Delta_{q',h}(X)$ lie on $X=\Delta^0_q=\Delta^0_{q'}$, then $\Delta_{q,h}(X)\cap\Delta_{q',h}(X)$ is the intersection of $2h-1$ hypersurfaces of $\mathbb{P}^{n+1}$, which is non--empty because of the assumption $h\leqslant \frac{n}{2}+1$. 
\end{proof}

In addition, when both $X\subset \mathbb{P}^{n+1}$ and the points $q,q'\in X$ are assumed to be general, the following holds.
\begin{lemma}\label{lemma:general} 
Let $n \geqslant 2$ be an integer and let $X \subset \mathbb{P}^{n+1}$ be a general hypersurface of degree $d \geqslant 2$. 
Then, for any integer $2 \leqslant h \leqslant \frac{n}{2} +1$ and for general $q,q'\in X$, there exists a general point $p\in X$ such that $q,q'\in V^{h}_p$.
\end{lemma}

\begin{proof}
Set $2 \leqslant h \leqslant \frac{n}{2} +1$ and consider the variety
$${\mathcal P}_h := \overline { \left.\left\{(q,q',p) \in X \times X \times X \right| q\neq q' \text{ and } p \in \Delta_{q,h} \cap \Delta_{q',h} \right\} }$$
endowed with the projections
$$
X\times X\stackrel{\pi_{12}}{\longleftarrow}\mathcal{P}_h \stackrel{\pi_{3}}{\longrightarrow}X.
$$
Thanks to Lemma \ref{lemma:polari}, the map $\pi_{12}$ is surjective. 
Let $\mathcal{Z}\subset {\mathcal P}_h$ be an irreducible component dominating $X\times X$.
Therefore, the proof of Lemma \ref{lemma:polari} gives that for any $(q,q')\in X\times X$, 
\begin{equation}\label{eq:dimDelta}
\dim\left(\Delta_{q,h}(X)\cap\Delta_{q',h}(X)\right)\geqslant n+2-2h \quad \text{and hence} \quad \dim \mathcal{Z}\geqslant 3n+2-2h.
\end{equation}
Moreover, given any point $p\in \pi_3(\mathcal{Z})$ and setting $Z_p := V_p^h \cap X$, we have
\begin{equation}\label{eq:ppp}
\left({\pi_3}_{|\mathcal{Z}}\right)^{-1}(p)\subseteq \pi_3^{-1} (p) \cong Z_p \times Z_p.
\end{equation}

In order to prove that for general $q,q'\in X$, there exists a general point $p\in X$ such that $q,q'\in V^{h}_p$, it is enough to prove that $\mathcal{Z}$ dominates $X$ via $\pi_3$.
We assume by contradiction that ${\pi_3}_{|\mathcal{Z}}$ is not dominant, i.e. $\dim {\pi_3}(\mathcal{Z}) < n  = \dim X$. 
If $p\in \pi_3(\mathcal{Z})$ is a general point, then \eqref{eq:dimDelta} and \eqref{eq:ppp} give
\begin{equation}\label{eq:part1}
\dim(Z_p\times Z_p)\geqslant \dim \left(\left({\pi_3}_{|\mathcal{Z}}\right)^{-1}(p) \right)= \dim\mathcal{Z}-\dim {\pi_3}(\mathcal{Z}) > 2n+2-2h.
\end{equation}
It follows that $\dim Z_p > n+1-h$ and hence $\dim V_p^h > n+2-h$.
Since the latter inequality contradicts Theorem \ref{theorem:dimV^h_p}, we conclude that ${\pi_3}_{|\mathcal{Z}}$ is dominant.
\end{proof}

\section{$k$--irrationality degree and connecting gonality of general hypersurfaces}\label{section:bounds} 

In this section we apply the previous results in order to bound the $k$--irrationality degree and the covering gonality of a very general hypersurface $X\subset \mathbb{P}^{n+1}$ of degree $d\geqslant 2n+2$.

According to Section \ref{section:cones}, we recall that if $V^h_p\subset \mathbb{P}^{n+1}$ is the cone of tangent lines of order $h$ at $p\in X$, we denote by $\Lambda^h_p$ a general hyperplane section of $V^h_p$. 
The link between the cones $V^h_p\subset \mathbb{P}^{n+1}$ and the invariants $\irr_k(X)$ and $\conngon(X)$ is expressed by the following result, which extends \cite[Proposition 2.12]{BCFS} to higher dimensional subvarieties of $X$. 

\begin{proposition}\label{proposition:Cone} 
Let $n\geqslant 3$ and $1\leqslant k\leqslant n-1$ be integers.
Let $X\subset \mathbb{P}^{n+1}$ be a very general hypersurface of degree $d\geqslant 2n+2$.
Suppose that for a general point $q\in X$, there exist a $k$--dimensional irreducible subvariety $Z\subset X$ containing $q$ and a dominant rational map $\varphi\colon Z\dashrightarrow \mathbb{P}^k$ of degree $c\leqslant d-3$. Then:
\begin{itemize}
  \item[(i)] there exists a point $p\in Z$ such that $Z \subset V_{p}^{d-c}\cap X$;
  \item[(ii)] the map $\varphi\colon Z\dashrightarrow \mathbb{P}^k$ of degree $c$ is the projection from $p$.
\end{itemize}
In particular, the image of $Z$ under $\varphi$ is a $k$--dimensional rational variety $R\subset \Lambda_{p}^{d-c}$.
\begin{proof}
The case $k=1$ is covered by  \cite[Proposition 2.12]{BCFS}, so we assume $2\leqslant k\leqslant n-1$.
Since $d\geqslant 2n+2$, we have that $c\leqslant d-3 < 2d-2n-1$.
Let $z\in Z$ be a general point and let $\ell\subset \mathbb{P}^k$ be a general line passing through $\varphi(z)$.
Consider the curve $C_{\ell}\subset \varphi^{-1}(\ell)$ which is the union of all irreducible components of curves in $Z$ which dominate $\ell$ via $\varphi$. 
We claim that $C_{\ell}$ is irreducible.
Indeed, if $C'$ and $C$ were two irreducible components, then $\gon(C_{\mathrm{red}})+\gon(C_{\mathrm{red}}')\leqslant \deg(\varphi_{|{C_{\ell}}})$.
Being $\varphi_{|{C_{\ell}}}\colon C_{\ell} \dashrightarrow \ell\cong \mathbb{P}^1$ a map of degree $c$, either $C_{\mathrm{red}}$ or $C_{\mathrm{red}}'$ would have gonality at most $\frac{c}{2}< d-n-\frac{1}{2}$.
By varying $\ell\subset \mathbb{P}^k$, $z\in Z$ and $q\in X$, we  conclude that $X$ is covered by curves of gonality smaller than $d-n$,
 which is impossible (cf. \cite[Theorem A]{BDELU}).
The same argument shows that $C_{\ell}$ is reduced.

Therefore, we may define a family $\mathcal{C}\stackrel{\pi}{\longrightarrow} U\subset \mathbb{G}(1,k)$ of curves with a map of degree $c$ to $\mathbb{P}^1$, where $U\cong\mathbb{P}^{k-1}$ parametrizes lines through $\varphi(z)$, and for any $[\ell]\in U$, the corresponding curve is $C_{\ell}$.
As we vary $q\in X$ (and hence $Z$), we may define a family of curves covering $X$, each endowed with a $c$--gonal map.
Thus \cite[Proposition 2.12]{BCFS} ensures that for general $[\ell]\in U$, there exists a point $x_\ell\in C_\ell$ such that $C_\ell\subset X\cap V^{d-c}_{x_\ell}$ and the degree $c$ map $\varphi_{|C_\ell}\colon C_\ell\dashrightarrow \ell\cong \mathbb{P}^1$ is the projection from $x_\ell$.

Next we need to show that all the points $x_\ell$ coincide with some fixed point $p\in Z$.
For this we consider the map $\psi\colon U\dashrightarrow Z\subset X$ sending $[\ell]\in U$ to the corresponding point $x_\ell$.
Since $U\cong \mathbb{P}^{k-1}$, the image of $\psi$ is unirational.
As $X$ does not contain rational curves (see e.g. \cite{Cl}), we conclude that $\psi(U)$ is a point $p\in Z$.
Thus $Z$ is covered by the curves $C_\ell\subset V^{d-c}_{p}$ for general $[\ell]\in U$, and the degree $c$ map $\varphi_{|C_\ell}\colon C_\ell\dashrightarrow \ell\cong \mathbb{P}^1$ is the projection from $p$.
The assertion follows.
\end{proof}
\end{proposition}

\begin{remark}\label{remark:dimension} 
Let $X$ be an irreducible projective variety of dimension $n$ and let $\mathcal{Z}\stackrel{\pi}\longrightarrow T$ be a family of $k$--dimensional subvarieties of  $X$. 
If $\mathcal{Z}\stackrel{\pi}\longrightarrow T$ is a \emph{covering family} (i.e. for general $q\in X$, there exists $t\in T$ such that $q\in Z_t=\pi^{-1}(t)$), then
$\dim (T)\geqslant n-k$.
Indeed, the  map $f\colon\mathcal{Z}\longrightarrow X$ must be dominant and hence $\dim (\mathcal{Z})=\dim (T)+k\geqslant n$.

If in addition $\mathcal{Z}\stackrel{\pi}\longrightarrow T$ is a \emph{connecting family} (i.e. for general $q,q'\in X$, there exists $t\in T$ such that $q,q'\in Z_t$), then $\dim (T)\geqslant 2n-2k$.
Indeed the map $\mathcal{Z}\times_T\mathcal{Z}\longrightarrow X\times X$ induced by  $f\colon\mathcal{Z}\longrightarrow X$  must be dominant, hence $\dim\left(\mathcal{Z}\times_T\mathcal{Z}\right)=2k+\dim (T)\geqslant 2n$.
\end{remark}

\begin{remark}\label{remark:nm}
If $X\subset \mathbb{P}^{n+1}$ and $Y\subset \mathbb{P}^{m+1}$ are very general hypersurfaces of degree $d$, with $n\leqslant m$, then
$$\irr_k(Y)\leqslant\irr_k(X)  \text{ for any }1\leqslant k\leqslant n \quad \text{and}\quad \conngon(Y)\leqslant\conngon(X).$$
Indeed, the section of $Y$ by a general $(n+1)$--plane of $\mathbb{P}^{m+1}$ is a very general hypersurface of $\mathbb{P}^{n+1}$.
\end{remark}

We can now prove Theorem \ref{theorem:irr_k}.

\begin{proof}[Proof of Theorem \ref{theorem:irr_k}]
When $k=n$, the assertion is covered by \cite[Theorem C]{BDELU} and $\irr_n(X)=d-1$.

If $k=n-1$, we claim that $\irr_{n-1}(X)\leqslant d-2$. 
Indeed, tangent hyperplane sections $Z=X\cap T_{p}X$ of $X$ are $(n-1)$--dimensional varieties of degree $d$ having a double point at $p$ (see Lemma \ref {lemma:3ple}), so that the projection from $p$ is a dominant rational map $Z\dashrightarrow \mathbb{P}^{n-1}$ of degree $d-2$.
On the other hand, suppose by contradiction that $\irr_{n-1}(X)=c\leqslant d-3$.
Proposition \ref{proposition:Cone} ensures that any $(n-1)$--dimensional subvariety $Z\subset X$ computing $\irr_{n-1}(X)$ is contained in $X\cap V^{d-c}_p$ for some $p\in X$.
Thanks to \eqref{eq:inequalities} and \cite[Theorem A]{BDELU}, we have $c\geqslant \irr_1(X)\geqslant d-n$, so that $3\leqslant d-c\leqslant n$.
Then Theorem \ref{theorem:dimV^h_p} gives that $\dim V^{d-c}_p=n+2-(d-c)$.
In order to cover $X$ by $(n-1)$--dimensional varieties cut out by the cones $V^{d-c}_p$, we must have that $\dim(X\cap V^{d-c}_p)=n+1-(d-c)\geqslant n-1$ and hence $c\geqslant d-2$, a contradiction.
Thus $\irr_{n-1}(X)=d-2$.

If $1\leqslant k\leqslant n-2$, we claim that $\irr_k(X)\leqslant d-3$.
To see this, we note that for any $p\in X$, Theorem \ref{theorem:dimV^h_p} ensures that $V^{3}_p$ is a cone in $T_{p}X\cong \mathbb{P}^{n}$ over a quadric $\Lambda^3_p\subset \mathbb{P}^{n-1}$ (cf. Section \ref{section:cones}). 
Then the variety $Z=X\cap V^{3}_p$ has dimension $n-2$ and the projection from $p$ is a dominant map $Z\dashrightarrow \Lambda^3_p$ of degree $d-3$ to a rational variety.
Thus $\irr_1(X)\leqslant \dots\leqslant\irr_{n-2}(X)\leqslant d-3$.

Finally, any $k$--dimensional subvariety $Z\subset X$ computing $c:=\irr_k(X)$ is contained in some $X\cap V^{d-c}_p$ by Proposition  \ref{proposition:Cone}.
As above, we deduce $d-c \leqslant n$ and for any $p \in X$, Theorem \ref{theorem:dimV^h_p} gives $\dim(X\cap V^{d-c}_p)=n+1-(d-c)$.
Thus, in order to cover $X$ by $k$--dimensional varieties in $X\cap V^{d-c}_p$, we must have that $n+1-(d-c)\geqslant k$, that is $c\geqslant d-1-n+k$.

For $k=n-2$, the latter inequality gives $\irr_{n-2}(X)\geqslant d-3$, so the assertion follows.
\end{proof}

Now, we prove Theorem \ref{theorem:conngon}.

\begin{proof}[Proof of Theorem \ref{theorem:conngon}]
By \cite[Lemma 2.2]{BCFS}, if $p\in X$ is a general point and $3\leqslant h\leqslant \min\{n+1,d\}$, then $\Lambda^h_p$ is a general complete intersection of type $(2,3,\dots,h-1)$ in $\mathbb{P}^{n-1}$.
If $\Lambda^h_p$ is a Fano variety, then it is rationally connected (see \cite{KMM}).
The canonical bundle of $\Lambda^h_p\subset \mathbb{P}^{n-1}$ is $\mathcal{O}_{\Lambda^h_p}(\sum_{i=2}^{h-1}i-n)$.
Therefore, $\Lambda^h_p$ is a Fano variety if and only if
\begin{equation}\label{eq:iff}
\sum_{i=2}^{h-1}i\leqslant n-1 \;\;\Longleftrightarrow \;\;\frac{h(h-1)}{2}-1\leqslant n-1\;\; \Longleftrightarrow \;\;h\leqslant\left\lfloor\frac{\sqrt{8n+1}+1}{2}\right\rfloor.
\end{equation}
We note that $d> n+1\geqslant \left\lfloor\frac{\sqrt{8n+1}+1}{2}\right\rfloor\geqslant 3$ and we assume hereafter $h:=\left\lfloor\frac{\sqrt{8n+1}+1}{2}\right\rfloor$, so that the general $\Lambda^h_p$ is a smooth, rationally connected, complete intersection, whose dimension is $n+1-h$.

Setting $Z_p:=V^h_p\cap X$, the projection $\varphi\colon Z_p\dashrightarrow \Lambda^h_p$ from $p$ has degree $d-h\leqslant d-3$.
Given two general points $q,q'\in Z_p$, let $D\subset \Lambda^h_p$ be a rational curve connecting $\varphi(q)$ and $\varphi(q')$, and let $C:=\varphi^{-1}(D)$.
By arguing as for the curves $C_{\ell}$ in the proof of Proposition \ref{proposition:Cone}, we deduce that $C$ is integral.
Then $C$ is an irreducible curve passing through two general points $q,q'\in Z_p$ endowed with a map $\varphi_{|C}\colon C\dashrightarrow D$ of degree $d-h$.
Thus $\conngon(Z_p)\leqslant d-h=d-\left\lfloor\frac{\sqrt{8n+1}+1}{2}\right\rfloor$.

Since $n \geqslant 4$, we have $h=\left\lfloor\frac{\sqrt{8n+1}+1}{2}\right\rfloor \leqslant \frac{n}{2} +1$. 
So Lemma \ref{lemma:general} ensures that for general $q,q'\in X$, there exists a general point $p\in X$ such that $q,q'\in Z_p$, i.e. the varieties $Z_p$ produce a connecting family.
Thus $\conngon(X)\leqslant \conngon(Z_p)\leqslant d-\left\lfloor\frac{\sqrt{8n+1}+1}{2}\right\rfloor$.
\end{proof}

Let us consider integers $n,d\geqslant 2$ and $2\leqslant h\leqslant \min\{n+1,d\}$. 
Before proving Theorem \ref{theorem:conngonRY}, we aim at introducing a suitable parameter space $\Theta^{n+1}_h$ for 4--tuples $(p,\ell_1,\ell_2, X)$, where $X\subset \mathbb{P}^{n+1}$ is a hypersurface of degree $d$, $p\in X$ and $\ell_1,\ell_2\subset V^h_{p,X}$ are lines having intersection multiplicity at least $h$ with $X$ at $p$. 

To this aim, we define $S_d := \mathbb{C} [x_0, \ldots, x_{n+1}]_d$ and $S_d^* := S_d \setminus \{0\}$ as in \eqref{eq:S_k}, and we set 
$$\mathbb{P} :=\mathbb{P}^{n+1}, \quad \mathbb{G} := \mathbb{G} (1,n+1)\quad \text{and}\quad N+1 := \dim_{\mathbb{C}}(S_d)={d+n+1\choose d}. $$
Let $\mathcal{P} \subset \mathbb{P}\times \mathbb{G}$ be the universal family of lines over $\mathbb{G}$, endowed with the projections
$ \mathbb{P} \stackrel {\pi_1} \longleftarrow  \mathcal{P} \stackrel {\pi_2}\longrightarrow \mathbb G $.
The morphism $\pi_1$ makes $\mathcal P$ a $\mathbb{P}^n$--bundle over $\mathbb{P}$, whereas $\pi_2$ makes $\mathcal P$ a $\mathbb{P}^1$-bundle over $\mathbb G$, so that $\dim (\mathcal P) =2 n+1$. 
Consider the fibred product
$$
\mathcal P \times_{\mathbb{P}} \mathcal P := \left\{\left(p, [\ell_1], [\ell_2]\right) \left|\ p \in \ell_1 \cap \ell_2 \right.\right\} \subset \mathbb{P} \times \mathbb{G} \times \mathbb{G}
$$
and its \emph{diagonal locus}
$$
\Delta:= \left\{(p, [\ell], [\ell]) \in \mathcal P \times_{\mathbb{P}} \mathcal P \left| \ p \in \ell \right.\right\} \cong \mathcal P.
$$
Let $\widetilde{\mathcal{P}}$ denote the blow-up of $\mathcal P \times_{\mathbb{P}} \mathcal P $ along $\Delta$ and let $\widetilde{\Delta}$ be the exceptional divisor. Thus,
$$
\dim (\widetilde{\mathcal{P}}) = 3n+1 = \dim (\widetilde{\Delta}) + 1.
$$
Moreover, as a set, we have
$$
\widetilde{\Delta}= \left\{(p, [\ell], [\Pi])  \left| \ p \in \ell \subset \Pi\right.\right\} \subset \mathbb{P} \times \mathbb{G} \times \mathbb{G}(2, n+1).\footnote{This fact follows from a standard argument, but we sketch it for the sake of completeness.
The exceptional divisor $\widetilde{\Delta}$ is the projectivization of the normal bundle of $\Delta$ in $\mathcal P\times_{\mathbb P}\mathcal P$. 
Let $\Gamma$ denote the algebraic set $\{(p,[\ell],[\Pi])|p\in\ell\subset \Pi\}$ on the right--hand side. 
There is an obvious morphism $\xi\colon \Gamma\longrightarrow \widetilde{\Delta}$, that maps a triple $(p,[\ell],[\Pi])\in \Gamma$ to the point of $\widetilde{\Delta}$ corresponding to the deformation of $(p,[\ell],[\ell])$ to a pair $(p,[\ell],[\ell'])$, where $\ell'$ moves in the pencil of lines passing through $p$ inside the plane $\Pi$. The map $\xi$ is clearly injective. It is also surjective, because it has an inverse $\eta\colon \widetilde{\Delta}\longrightarrow \Gamma$ defined as follows. A point $x\in\widetilde{\Delta}$ that lies over $(p,[\ell],[\ell])$ corresponds to a deformation $\{(p_t,[\ell_{1,t}],[\ell_{2,t}])\}$ of $(p,[\ell],[\ell])$, with $t$ moving in a disc with centre $0$. Then $\eta$ associates to $x$ the point $(p,[\ell],[\Pi])\in\Gamma$, where $\Pi$ is the flat limit of the plane spanned by $\ell_{1,t}$ and $\ell_{2,t}$, when $t\to 0$. This shows that $\Gamma$ and $\widetilde{\Delta}$ are set--theoretically the same.}
$$
Given a polynomial $F \in S_d^*$, let $X_F:= V(F) \subset \mathbb P$ denote its vanishing locus.
Then we define the variety ${\Theta}^{n+1}_h \subset \widetilde{\mathcal{P}} \times \mathbb{P}(S_d)$ as 
\begin{equation}\label{Theta_h}
{\Theta}^{n+1}_h :=\overline{\left\{ (p, [\ell_1], [\ell_2], [F])\in \widetilde{\mathcal{P}} \times \mathbb{P}(S_d) \left| \begin{array}{l}\ell_1\neq \ell_2 \text{ and for }1 \leqslant i \leqslant 2,\\ \text{ either } \ell_i\subset X_F \text{ or } X_F \cdot \ell_i \geqslant hp \end{array} \right.\right\}}.
\end{equation}

\begin{lemma}\label{lem:quadratinor} For any $2 \leqslant h \leqslant \min \{n+1,d\}$, ${\Theta}^{n+1}_h$ is smooth, irreducible, of dimension
$3n + 2 + N - 2h$, dominating both $\widetilde{\mathcal{P}}$ and $\mathbb{P}(S_d)$ via the  projection maps 
\[
 \widetilde{\mathcal{P}} \stackrel {\Psi} \longleftarrow  {\Theta}^{n+1}_h \stackrel {\Phi}\longrightarrow  \mathbb{P}(S_d).
 \]
\end{lemma}
\begin{proof} Let $(p, [\ell_1], [\ell_2])\in \widetilde{\mathcal{P}} \setminus \widetilde{\Delta}$ and $F\in S_d^*$.
Requiring that $(p, [\ell_1], [\ell_2], [F])\in {\Theta}^{n+1}_h$ amounts to impose $2h-1$ independent linear conditions to $F$, corresponding to the conditions $\ell_1\cdot X_F\geqslant hp$ and $\ell_2\cdot X_F\geqslant hp$.  

We claim that the same happens at $(p, [\ell], [\Pi])\in \widetilde{\Delta}$, when we require $(p, [\ell], [\Pi], [F])\in {\Theta}^{n+1}_h$. 
In fact,  choose affine coordinates $(\eta,\zeta)$ on $\Pi$ such that $p=(0,0)$ and $\ell= V(\eta)$. Write $F|_{\Pi}=F_0+F_1+\dots + F_d$, where $F_i=\sum_{0 \leqslant j  \leqslant i} a_{i,j}\eta^{i-j}\zeta^j$ is a homogeneous polynomial of degree $i$.
Imposing the condition $X_F \cdot \ell \geqslant hp$ gives $a_{i,i}=0$, for any $0 \leqslant i  \leqslant h-1$. Now, consider a general line  $\ell' := V(\eta - t\zeta)$ in $\pi$ 
through $p=(0,0)$. 
Imposing the condition $X \cdot \ell'\geqslant hp$ and letting $\ell'$ approach $\ell$, i.e. letting $t$ approach zero, gives $a_{i,i-1}=0$ for any $1 \leqslant i \leqslant h-1$. 
Therefore, there are again $2h-1$ independent conditions for the coefficients of $F$ in order to have $(p, [\ell], [\Pi],[F])\in {\Theta}^{n+1}_h$.

Hence the projection $\Psi\colon{\Theta}^{n+1}_h\longrightarrow \widetilde{\mathcal{P}}$ is onto, and its fibers are parameterized by $(2h-1)$--codimensional linear subspaces of $\mathbb{P}(S_d)$. 
Therefore, ${\Theta}^{n+1}_h$ is smooth, irreducible, of dimension
\[
\dim ({\Theta}^{n+1}_h) = \dim (\widetilde{\mathcal{P}})+ \dim (\mathbb{P}(S_d)) - (2h-1)=3n+1+N-2h+1.
\]
The surjectivity of $\Phi\colon \Theta^{n+1}_h\longrightarrow  \mathbb{P}(S_d)$ is clear; indeed, for any $F \in S_d^*$ and any $p \in X_F$, we have that $\dim (V_{p,X_F}^h) \geqslant n+2 -h >0$, as $h \leqslant n+1$ by assumption.  
\end{proof}

For a general polynomial $F\in S_d^*$, we set
\begin{equation*}
{\Theta}_{h,F}^{n+1}:=\Phi^{-1}([F]),
\end{equation*} 
which is smooth, equidimensional, of dimension $3n+2-2h$. 

\smallskip
Now, we argue as in \cite[Proof of Theorem 2.3]{RY} and we prove Theorem \ref{theorem:conngonRY}.
\begin{proof}[Proof of Theorem \ref{theorem:conngonRY}]
Let $X\subset \mathbb{P}^{n+1}$ be a very general hypersurface of degree $d\geqslant 2n+2$ and let $h\in \mathbb{N}$ such that $\conngon(X)=d-h$.
It follows from Theorem \ref{theorem:conngon} that $h\geqslant \left\lfloor\frac{1+\sqrt{8n+1}}{2}\right\rfloor$.
We note that if $h= \left\lfloor\frac{1+\sqrt{8n+1}}{2}\right\rfloor$, we are done because $\left\lfloor\frac{1+\sqrt{8n+1}}{2}\right\rfloor\leqslant \left\lfloor\frac{\sqrt{16n+25}-3}{2}\right\rfloor$ for any $n\geqslant 4$. 
Hence we assume hereafter 
\begin{equation}\label{eq:h}
h> \left\lfloor\frac{1+\sqrt{8n+1}}{2}\right\rfloor. 
\end{equation}

Given two general points $x_1,x_2\in X$, there exists an irreducible curve $C\subset X$ containing $x_1$ and $x_2$ such that $\gon(C)=\conngon(X)=d-h$.
Since $h\geqslant 3$, by Proposition \ref{proposition:Cone} there exists a point $p\in X$ such that $C\subset V^h_{p}$, and the projection $\pi_p\colon V^h_{p}\dashrightarrow \Lambda_{p}^h$ from $p$ maps $C$ to a rational curve $D\subset \Lambda_{p}^h$.
In particular, if $\ell_1,\ell_2\subset V^h_{p}$ denote the lines connecting the vertex $p$ to the points $x_1,x_2\in X$, respectively, then $D$ passes through the corresponding points $\ell_1\cap \Lambda_{p}^h$ and $\ell_2\cap \Lambda_{p}^h$.
So, for any integer $t\geqslant n+1$, we consider the variety $\Theta^t_h$ defined as in \eqref{Theta_h} (using the obvious modifications $\mathbb{P}=\mathbb{P}^t$, $\mathbb{G}=\mathbb{G}(1,t)$, etc.), and we introduce the locus $\mathcal{R}_{t}\subset {\Theta}^t_{h}$ given by 
\begin{equation}\label{Rh}
\mathcal{R}_{t}:= \overline{\left\{ (p, [\ell_1], [\ell_2], [F])\in {\Theta}^t_{h} \left| \begin{array}{l}\ell_1\neq \ell_2 \text{ and }\exists \text{ a rational curve } D\subset \Lambda_{p,X_F}^h \text{ passing}\\ \text{through the points } \ell_1\cap \Lambda_{p,X_F}^h \text{ and } \ell_2\cap \Lambda_{p,X_F}^h\end{array}\right.\right\}}.
\end{equation} 
In particular, we are interested in the case $t=n+1$.

It follows from \cite[Lemma 2.2]{BCFS} that for $F\in S_d^*$ general and $p \in X_F$ general, the variety $ \Lambda_{p}^h= \Lambda_{p,X_F}^h$ is a general complete intersection of type $(1,1,2,\dots,h-1)$ in $\mathbb{P}^{n+1}$.
Thus its canonical bundle is isomorphic to $\mathcal{O}_{\Lambda^h_p}(\sum_{i=2}^{h-1}i-n)$, which is effective by \eqref{eq:iff} and \eqref{eq:h}.
In particular, $ \Lambda_{p}^h$ is not covered by rational curves, so that $\mathcal{R}_{n+1}$ consists of (at most) countably many proper closed subsets of ${\Theta}^{n+1}_{h}$.

Let $F\in S^*_d$ be the polynomial defining the very general hypersurface $X\subset \mathbb{P}^{n+1}$.
According to the discussion above, for general $x_1,x_2\in X$, we may find (at least) one 4--tuple $(p, [\ell_1], [\ell_2], [F])\in\mathcal{R}_{n+1}$, where $\ell_i=\langle p,x_i\rangle$ for $i=1,2$.
Since $\ell_i\cap X$ consists of finitely many points, as we vary the pair $(x_1,x_2)\in X\times X$, the corresponding 4--tuples $(p, [\ell_1], [\ell_2], [F])$ describe a subset of $\mathcal{R}_{n+1}\cap \Theta_{h,F}^{n+1}$ having dimension at least $2n$.
In particular, $\dim \left(\mathcal{R}_{n+1}\cap \Theta_{h,F}^{n+1}\right)\geqslant 2n$ and as $F\in S^*_d$ is very general, we deduce that $\dim \mathcal{R}_{n+1}\geqslant 2n+N$.
Therefore,
\begin{equation}
\label{eq:nec2}
\codim_{\Theta^{n+1}_{h}}\mathcal{R}_{n+1}\leqslant n+2-2h.
\end{equation}
We point out that for any subfamily $\mathcal{F}\subset \Theta^{n+1}_{h}$ such that $(p, [\ell_1], [\ell_2], [F])\in \mathcal{F}$, we have $\codim_{\mathcal{F}}\left(\mathcal{R}_{n+1}\cap \mathcal{F}\right)\leqslant \codim_{\Theta^{n+1}_{h}}\mathcal{R}_{n+1}$. 
Hence \eqref{eq:nec2} gives
\begin{equation}
\label{eq:nec3}
\codim_{\mathcal{F}}\left(\mathcal{R}_{n+1}\cap \mathcal{F}\right)\leqslant n+2-2h.
\end{equation}

\smallskip
We construct a subfamily $\mathcal{F}\subset \Theta^{n+1}_{h}$ with $(p, [\ell_1], [\ell_2], [F])\in\mathcal{F}$, as follows.
Let 
$$m:=\frac{h(h-1)}{2}-2$$ 
and let $(p', [\ell'_1], [\ell'_2], [F'])\in \Theta^{m+2}_h$ be a $4$--tuple such that $Y':=V(F')\subset \mathbb{P}^{m+2}$ is a very general hypersurface of degree $d$, $p'\in Y'$ is a very general point, $\ell'_1$ is very general among lines in $V^h_{p',Y'}$ passing through $p'$, and $\ell'_2\neq \ell'_1$.
Moreover, we deduce from \eqref{eq:h} that $m+2\geqslant n+1$.

Let $M\geqslant m+2\geqslant n+1$ and let $(p'', [\ell''_1], [\ell''_2], [F''])\in \Theta^{M}_h$, where $Y'':=V(F'')\subset \mathbb{P}^M$ is a hypersurface of degree $d$ such that $X$ is a $(n+1)$--plane section and $Y'$ is a $(m+2)$--plane section, with $p=p'=p''$, $\ell_1=\ell'_1=\ell''_1$, and $\ell_2=\ell'_2=\ell''_2$.

Now, for any $r\geqslant n+1$, let $Z_r\subset \mathrm{Hom}\left(\mathbb{P}^r, \mathbb{P}^M\right)$ be the set of parameterized $r$--planes in $\mathbb{P}^M$ containing the plane $\langle \ell_1,\ell_2\rangle$, and let $Z'_r\subset Z_r$ be the subset of parameterized $r$--planes $\Lambda\subset\mathbb{P}^M$ such that $(p, [\ell_1], [\ell_2], [F_\Lambda])\in \mathcal{R}_{r}$, where $F_\Lambda$ is a polynomial defining the section of $Y''$ by $\Lambda$ as a hypersurface in $\Lambda$.\footnote{As in \cite{RY}, we consider the locus $Z_r\subset \mathrm{Hom}\left(\mathbb{P}^r, \mathbb{P}^M\right)$---rather than its counterpart in $\mathbb{G}\left(r,M\right)$---because we need to fix homogeneous coordinates $[y_0:\ldots:y_r]$ on each $r$--plane $\Lambda\subset \mathbb{P}^M$, in order to define properly the polynomial $F_{\Lambda}\in \mathbb{C}[y_0,\dots,y_r]_d\smallsetminus\{0\}$ (up to scalar multiplication).}

We point out that for any $\Lambda\in Z_r$, we have $(p, [\ell_1], [\ell_2], [F_\Lambda])\in  \Theta^{r}_h$. 
To see this fact, consider the hypersurface $Y:=V(F_{\Lambda})=\Lambda\cap Y''$ of degree $d$ in $\Lambda\cong \mathbb{P}^{r}$.
For $i=1,2$, we have $\ell_i\subset \Lambda$, so the intersection schemes $\ell_i\cdot Y$ and $\ell_i\cdot Y''$ are supported on the same $0$--cycle of degree $d$, i.e. $\mult_q( \ell_i\cdot Y)=\mult_q( \ell_i\cdot Y'')$ for any $q\in Y\cap \ell_i$.
In particular, $Y\cdot \ell_i\geqslant hp$ for $i=1,2$, so that $(p, [\ell_1], [\ell_2], [F_\Lambda])\in  \Theta^{r}_h$.

As in \cite{RY}, let $\mathcal{F}$ be the image of $Z_{n+1}$ in $\Theta^{n+1}_h$ under the map sending a $(n+1)$--plane $\Lambda\in Z_{n+1}$ to the point $(p, [\ell_1], [\ell_2], [F_\Lambda])\in  \Theta^{n+1}_h$.
Thus $\mathcal{R}_{n+1}\cap \mathcal{F}$ is the image of $Z'_{n+1}$.
According to \eqref{eq:nec3}, we have 
\begin{equation}
\label{eq:nec4}
\codim_{Z_{n+1}}Z'_{n+1}\leqslant n+2-2h.
\end{equation}
Let $\varepsilon_{r}:= \codim_{Z_{r}}Z'_{r}$. 
Since $Y'\subset \mathbb{P}^{m+2}$ and $p\in Y'$ are very general, then $\Lambda^h_{p,Y'}$ is a smooth complete intersection of type $(1,1,2,\dots,h-1)$ in $\mathbb{P}^{m+2}$ by \cite[Lemma 2.2]{BCFS}. 
Hence its canonical bundle is $\mathcal{O}_{\Lambda^h_{p,Y'}}(\sum_{i=2}^{h-1}i-m-1)$, which is trivial by the choice of $m$. 
In particular, $\Lambda^h_{p,Y'}$ is not covered by rational curves and as $\ell_1\subset V^h_{p,Y'}$ is a very general line through $p$, there are no rational curves of $\Lambda^h_{p,Y'}$ passing through the point $\ell_1\cap \Lambda^h_{p,Y'}$.
Thus $(p, [\ell_1], [\ell_2], [F'])\not\in \mathcal{R}_{m+2}$ and $\varepsilon_{m+2}\geqslant 1$.

Applying \cite[Proposition 2.5]{RY}, we obtain
\begin{equation*}
\varepsilon_{m+1}=\codim_{Z_{m+1}}Z'_{m+1}\geqslant \varepsilon_{m+2}+1\geqslant 2,
\end{equation*}
and by recursion
\begin{equation*}
\varepsilon_{n+1}=\codim_{Z_{n+1}}Z'_{n+1}\geqslant m-n+ 2.
\end{equation*}
By \eqref{eq:nec4}, we must have $m-n+ 2\leqslant n+2-2h$, and as $m:=\frac{h(h-1)}{2}-2$, we deduce
$$
\frac{h(h-1)}{2}-n\leqslant n+2-2h,\quad\text{so that}\quad h\leqslant \frac{-3+\sqrt{16n+25}}{2}.
$$
Thus the connecting gonality of $X$ satisfies $\conngon(X)\geqslant d- \left\lfloor\frac{-3+\sqrt{16n+25}}{2}\right\rfloor$.

The final part of the statement is achieved by using \eqref{eq:irr&covgon} and noting that $\left\lfloor\frac{-1+\sqrt{16n+1}}{2}\right\rfloor=\left\lfloor\frac{-3+\sqrt{16n+25}}{2}\right\rfloor$ if and only if $n$ belongs to the set 
$$\left\{\left.4a^2+3a,4a^2+5a,4a^2+5a+1,4a^2+7a+2,4a^2+9a+4,4a^2+11a+6\right|a\in\mathbb{N}\right\}.$$
\end{proof}

Finally, we discuss the values of $\conngon(X)$, when the hypersurface $X$ has small dimension.

\begin{example}
\label{example:small n}
Let $X\subset \mathbb{P}^{n+1}$ be a very general hypersurface of degree $d\geqslant 2n+2$, with $1 \leqslant n \leqslant 16$ and $n\neq 9,13,14$.\\
\textbf{Case \emph{n}$=$1.} When $X$ is a plane curve, $\conngon(X)$ equals the gonality of $X$, which is $\gon(X)=d-1$ (cf. \cite[Teorema 3.14]{Ci2}).\\
\textbf{Case \emph{n}$=$2.} The connecting gonality of very general surfaces $X\subset \mathbb{P}^3$ of degree $d\geq 5$ is computed by tangent hyperplane sections $X\cap T_pX$, so that $\conngon(X)=d-2$ (see e.g. \cite{B2}).\\
\textbf{Case \emph{n}$=$3.} When $n=3$, we have $\conngon(X)=d-2$. 
To see this fact, notice that $\conngon(X)\leqslant d-2$ by Remark \ref{remark:nm} and case $n=2$ above.
On the other hand, $\conngon(X)\geqslant \covgon(X)= d-3$ by \eqref{eq:inequalities} and \eqref{eq:irr&covgon}.
Suppose by contradiction that there exists a connecting family $\mathcal{C}\stackrel{\pi}{\longrightarrow} T$ of $(d-3)$--gonal curves.
Then Proposition \ref{proposition:Cone} ensures that the general curve $C_t:=\pi^{-1}(t)$ lies on $X\cap V^3_p$ for some $p\in X$.
By Theorem \ref{theorem:dimV^h_p}, the varieties $Z_p:=X\cap V^3_p$ are curves and, as we vary $p\in X$, we obtain a 3--dimensional family.
However, according to Remark \ref{remark:dimension}, the family $\mathcal{C}\stackrel{\pi}{\longrightarrow} T$ should have dimension at least 4, a contradiction.\\
\textbf{Cases 4\!\:$\leqslant$\!\:\emph{n}\!\:$\leqslant$\!\:16 with \emph{n}$\neq$9,\!\:13,\!\:14.}
For all these values of $n$, we may apply Theorems \ref{theorem:conngon} and \ref{theorem:conngonRY}, and the bounds included therein coincide.
Thus 
$$
\conngon(X)=\left\{\begin{array}{ll}
d-3 & \text{if } n=4,5   \\
d-4 & \text{if } n=6,7,8 \\   
d-5 & \text{if } n=10,11,12 \\
d-6 & \text{if } n=15,16. \\
\end{array}\right.
$$
\end{example}

\section*{Acknowledgements}

We would like to thank Lawrence Ein, Enrico Fatighenti and Francesco Russo for helpful discussions.

\end{document}